\documentclass{amsart}

\usepackage{amsmath}
\usepackage{amsthm}
\usepackage{amssymb}
\usepackage{xspace}
\usepackage{color}
\usepackage{mathtools}
\usepackage{enumerate}
\usepackage{tikz}

\usetikzlibrary{matrix,arrows}

\theoremstyle{plain}
\newtheorem*{definition}{Definition}
\newtheorem{theorem}{Theorem}
\newtheorem{lemma}[theorem]{Lemma}
\newtheorem{corollary}[theorem]{Corollary}
\newtheorem{proposition}[theorem]{Proposition}

\theoremstyle{definition}

\newtheorem{example}[theorem]{Example}

\newcommand{\comment}[1]{}
\newcommand{\rea}{\ensuremath{\mathbf{R}}\xspace}
\newcommand{\nat}{\ensuremath{\mathbf{N}}\xspace}
\newcommand{\probs}{\ensuremath{\mathcal{P}}\xspace}
\newcommand{\intd}[1]{\,\mathrm{d}#1 \,}
\newcommand{\closure}[1]{\ensuremath{\overline{#1}}\xspace}
\newcommand{\dimh}{\ensuremath{\dim_{\mathrm{H}}}\xspace}
\newcommand{\dimf}{\ensuremath{\dim_{\mathrm{F}}}\xspace}
\newcommand{\dimfm}{\ensuremath{\dim_{\mathrm{FM}}}\xspace}
\newcommand{\dimfc}{\ensuremath{\dim_{\mathrm{FC}}}\xspace}

\newcommand\restr[2]{{								
  \left.\kern-\nulldelimiterspace						
  #1										
  \vphantom{\big|}								
  \right|_{#2}									
  }}

\DeclareMathOperator{\supp}{supp}
\DeclareMathOperator{\re}{Re}
\DeclareMathOperator{\sinc}{sinc}
\DeclarePairedDelimiter{\ceil}{\lceil}{\rceil}
\DeclarePairedDelimiter{\floor}{\lfloor}{\rfloor}

\begin{document}
\title{On the Fourier dimension and a modification}

\author{Fredrik Ekstr\"om}
\address{Centre for Mathematical Sciences\\ Lund University\\ Box 118\\ 22 100 Lund\\ SWEDEN}
\email{fredrike@maths.lth.se}
\author{Tomas Persson}
\email{tomasp@maths.lth.se}
\author{J\"org Schmeling}
\email{joerg@maths.lth.se}

\begin{abstract}
We give a sufficient condition for the Fourier dimension of a countable union of
sets to equal the supremum of the Fourier dimensions of the sets in the union,
and show by example that the Fourier dimension is not countably stable 
in general. A natural approach to finite stability
of the Fourier dimension for sets would be to try to prove that the
Fourier dimension for measures is finitely stable, but we give an example
showing that it is not in general. We also describe some situations where the
Fourier dimension for measures is stable or is stable
for all but one value of some parameter. Finally we propose a way 
of modifying the definition of the Fourier dimension so that it becomes
countably stable, and show that for each $s$ there is a class of sets
such that a measure has modified Fourier dimension greater than or
equal to $s$ if and only if it annihilates all sets in the class.
\end{abstract}

\maketitle

\section{Introduction}
Let $A$ be a Borel subset of $\rea^d$. One way to prove a lower bound for the
Hausdorff dimension of $A$ is to consider integrals of the form
	$$
	I_s(\mu) = \iint |x - y|^{-s} \intd{\mu}(x) \intd{\mu}(y);
	$$
if $\mu$ is a Borel measure such that $\mu(A) > 0$ and $I_s(\mu) < \infty$
for some $s$, then $\dimh A \geq s$.
For a finite Borel measure $\mu$, the
\emph{Fourier transform} is defined as
	$$
	\widehat\mu(\xi) = \int e^{-2\pi i \xi \cdot x} \intd{\mu}(x),
	$$
where $\xi \in \rea^d$ and $\cdot$ denotes the Euclidean inner product.
It can be shown \cite[Lemma~12.12]{mattila} that if $\mu$ has compact
support then
	$$
	I_s(\mu) = \text{const.}(d, s) \int |\widehat\mu(\xi)|^2 |\xi|^{s - d} \intd{\xi}
	$$
for $0 < s < d$, and thus $I_{s_0}(\mu)$ is finite if $\widehat\mu(\xi) \lesssim |\xi|^{-s/2}$
for some $s > s_0$ (here and in the remainder, $f(\xi) \lesssim g(\xi)$ means that
there exists a constant $C$ such that $|f(\xi)| \leq C|g(\xi)|$ for all $\xi$). This
motivates defining the \emph{Fourier dimension} of $A$ as
	$$
	\dimf A = \sup \left\{ s \in [0, d]; \, \widehat\mu(\xi) \lesssim |\xi|^{-s/2}, \,
	\mu \in \probs(A) \right\},
	$$
where $\probs(A)$ denotes the set of Borel probability
measures on $\rea^d$ that give full measure to $A$ (it would not make any difference if the supremum
was taken only over such measures with compact support, for if
$\widehat\mu(\xi) \lesssim |\xi|^{-s / 2}$
then there is a compactly supported probability measure
$\nu$ that is absolutely continuous with respect to $\mu$ and satisfies
$\widehat\nu(\xi) \lesssim |\xi|^{-s / 2}$ by Lemma~\ref{smoothcutlemma} below).
Thus the Fourier dimension is a lower bound for the Hausdorff
dimension. The Fourier dimension of a finite Borel measure $\mu$ on $\rea^d$ is defined as
	$$
	\dimf \mu = \sup\left\{ s \in [0, d]; \, \widehat\mu(\xi) \lesssim |\xi|^{-s/2} \right\},
	$$
or equivalently
	$$
	\dimf \mu = \min \left(
	d, \,\,
	\liminf_{|\xi| \to \infty} \frac{-2 \log |\widehat\mu(\xi)|}{\log|\xi|}
	\right),
	$$
so that
	$$
	\dimf A = \sup\left\{\dimf \mu; \, \mu \in \mathcal{P}(A) \right\}.
	$$

If $A \subset B$ then $\mathcal{P}(A) \subset \mathcal{P}(B)$ and hence
	$$
	\dimf (A) = \sup \{ \dimf \mu; \, \mu \in \mathcal{P}(A) \}
	\leq
	\sup \{ \dimf \mu; \, \mu \in \mathcal{P}(B) \}
	= \dimf(B),
	$$
showing that the Fourier dimension is monotone. It seems not to be previously known
whether the Fourier dimension is stable under finite or countable unions,
that is, whether
	\begin{equation}	\label{setstab}
	\dimf\left(\bigcup_k A_k \right) = \sup_k \, \dimf A_k,
	\end{equation}
where $\{ A_k \}$ is a finite or countable family of sets. The inequality
$\geq$ follows from the monotonicity, but there might be sets for which
the inequality is strict. In Section~\ref{sec2} we show that \eqref{setstab}
holds if for each $n$ the intersection $A_n \cap \closure{\bigcup_{k \neq n} A_k}$ has
small ``modified Fourier dimension'' (defined below), and in particular if all such
intersections are countable. We also give an example of a countably infinite family
of sets such that \eqref{setstab} does \emph{not} hold.

This still leaves open the question of \emph{finite} stability. The most
straightforward approach would be to prove a corresponding stability for
the Fourier dimension of measures, namely that
	\begin{equation}	\label{measstab}
	\dimf (\mu + \nu) = \min(\dimf \mu, \, \dimf \nu).
	\end{equation}
From this one could derive the finite stability for sets, using that
any probability measure on $A \cup B$ is a convex combination of
probability measures on $A$ and $B$. The inequality $\geq$
always holds in \eqref{measstab} since the set of functions that
are $\lesssim |\xi|^{-s/2}$ is closed under finite sums, but we give an example
in Section~\ref{sec3} showing that strict inequality can occur. We also
describe some situations in which \eqref{measstab} \emph{does} hold ---
this seems to be the typical case.

To achieve countable stability, we consider the following modification
of the Fourier dimension.

\begin{definition}
The \emph{modified Fourier dimension} of a Borel set $A \subset \rea^d$
is defined as
	$$
	\dimfm A = \sup\left\{\dimf \mu; \, \mu \in \probs(\rea^d), \,
	\mu(A) > 0 \right\},
	$$
and the modified Fourier dimension of a finite Borel measure $\mu$ is
defined as
	$$
	\dimfm \mu = \sup\left\{ \dimf \nu; \, \nu \in \probs(\rea^d), \,
	\mu \ll \nu \right\},
	$$
where $\ll$ denotes absolute continuity.
\end{definition}
Thus
	$$
	\dimfm A = \sup \left\{ \dimfm \mu; \, \mu \in \probs(A) \right\}.
	$$
In Section~\ref{sec4}, we investigate some basic properties of the modified
Fourier dimension, and give examples to show that it is different
from the usual Fourier dimension and the Hausdorff dimension.

In Section~\ref{sec5}, we show that if $\mu$ annihilates all
the common null sets for the measures that have modified Fourier dimension
greater than or equal to $s$, then $\dimfm \mu \geq s$.
Other classes of measures that can be characterised by their null sets in this way are the measures
that are absolutely continuous to some fixed measure, and, less trivially,
the measures $\mu \in \probs([0, 1])$ such that
$\lim_{|\xi| \to \infty} \widehat\mu(\xi) = 0$ (see \cite{lyons}). A necessary condition
for such a characterisation to be possible is that the class of measures be
a \emph{band}, meaning that any measure that is absolutely continuous to some
measure in the class lies in the class. The
definition of the modified Fourier dimension is natural from this point
of view, since the class of measures that have modified Fourier dimension
greater than or equal to $s$ is the smallest band that includes the measures
that have (usual) Fourier dimension greater than or equal to $s$.

\subsection{Previous work}
Here we mention briefly some of the previous work related to the Fourier dimension that we
are aware of. A \emph{Salem set} is a set whose Fourier dimension equals its Hausdorff dimension.

Salem \cite{salem} showed  				
that for each $s \in (0, 1)$ there is a compact Salem subset of $[0, 1]$
of dimension $s$ (although this was shown for the restriction of the
Fourier transform to the integers).
An \emph{explicit} example of a Salem set
of any prescribed dimension in $(0, 1)$ is given by the set of
\emph{$\alpha$-well approximable numbers}, namely the set
	$$
	E(\alpha) =
	\bigcap_{n = 1}^\infty \bigcup_{k = n}^\infty \left\{ x \in [0, 1]; \, \| kx \| \leq k^{-(1 + \alpha)} \right\},
	$$
where $\| \cdot \|$ denotes the distance to the nearest integer. By a theorem
of Jarn\'{\i}k \cite{jarnik} and Besicovitch \cite{besicovitch} the set $E(\alpha)$
has Hausdorff dimension $2 / (2 + \alpha)$ for $\alpha > 0$, and
Kaufman \cite{kaufman2} showed that there is a measure in $\probs(E(\alpha))$ with Fourier
dimension $2 / (2 + \alpha)$ (see also Bluhm's paper \cite{bluhm1}).

It was shown by Kaufman \cite{kaufman1} that for any $C^2$-curve $\Gamma$ in $\rea^2$ with
positive curvature and any $s \subset (0, 1)$, there is a compact Salem
set $S \subset \Gamma$ of dimension $s$. From this it can be deduced \cite[Proposition 1.1]{FOS}
that for any $s \in [0, 1]$ there is a continuous function $[0, 1] \to \rea$
whose graph has Fourier dimension $s$. Fraser, Orponen and Sahlsten \cite{FOS}
proved that the graph of any function $[0, 1] \to \rea$ has \emph{compact} Fourier dimension 
(defined below) less than or equal to $1$, and that the set of continuous functions
$[0, 1] \to \rea$ whose graphs have Fourier dimension $0$ is residual with respect to the supremum norm among
all continuous functions $[0, 1] \to \rea$.

Kahane showed that images of compact sets under Brownian motion and fractional Brownian motion 
are almost surely Salem sets, see \cite{kahane2}. It was shown by Fouch\'e and Mukeru \cite{FM}
that the level sets of fractional Brownian motion are almost surely Salem sets (in the special
case of Brownian motion this follows from a result of Kahane, see \cite[Section 3.2]{FM}).

Jordan and Sahlsten \cite{JS} showed that Gibbs measures of Hausdorff dimension greater than
$1 / 2$ (satisfying a certain condition on the Gibbs potential) for the Gauss map
$x \mapsto 1 / x \pmod{1}$ have positive Fourier dimension.

Wolff's book \cite{wolff} about harmonic analysis discusses some applications of the
Fourier transform to problems in geometric measure theory.

\subsection{Some remarks}
It is not so difficult to see that the Fourier dimension for measures is
invariant under translations and invertible linear transformations, and
thus the Fourier dimension and modified Fourier dimension for sets are
invariant as well.

For any finite Borel measure $\mu$ on $\rea^d$,
	$$
	\lim_{T \to \infty} \frac{1}{(2T)^d} \int_{[-T, T]^d}
	\left| \widehat\mu(\xi) \right|^2 \intd{\xi}
	=
	\sum_{x \in \rea^d} \mu(\{ x \})^2
	$$
(this is a variant of Wiener's lemma). If $\mu$ has an atom it is thus
not possible that $\lim_{|\xi| \to \infty} \widehat\mu(\xi) = 0$,
so $\dimf \mu = 0$, and also $\dimfm \mu = 0$ since
$\nu$ has an atom whenever $\mu \ll \nu$. It follows that $\dimf A = \dimfm A = 0$
for any countable set $A \subset \rea^d$.

Suppose next that $A$ is a countable union of $k$-dimensional hyperplanes
in $\rea^d$ with $k < d$. If $\mu$ gives positive measure to $A$, then there
must be a hyperplane $P$ such that $\mu(P) > 0$.
But then the projection of $\mu$ onto any line $L$ that goes through the origin and is
orthogonal to $P$ has an atom, so $\widehat\mu$ does not decay along $L$.
This shows that $\dimf A = \dimfm A = 0$. Thus for example a line segment in $\rea^2$
has Fourier dimension $0$ even though an interval in \rea has Fourier dimension $1$.

From a special case of a theorem by Davenport, Erd\H{o}s and LeVeque \cite{DEL}, it can
be derived \cite[Corollary~7.4]{QR} that if $\mu$ is a
probability measure on \rea such that $\widehat\mu(\xi) \lesssim |\xi|^{-\alpha}$
for some $\alpha > 0$, then $\mu$-a.e.~ $x$ is normal to any base (meaning that
$(b^k x)_{k = 0}^\infty$ is uniformly distributed mod $1$ for any $b \in \{ 2, 3, \ldots\}$).
Thus if $A \subset \rea$ does not contain any number that is normal to all
bases, then $\dimf A = \dimfm A = 0$. In particular this applies to
the middle-third Cantor set, since it consists of numbers that do not have
any $1$ in their ternary decimal expansion and hence are not normal to
base $3$.

\subsection{Other variants of the Fourier dimension}
One alternative way of defining the Fourier dimension of a Borel set
$A \subset \rea^d$ is to require the measure in the definition to
give full measure to a compact subset of $A$, rather than to $A$ itself.
This variant, which will here be called the \emph{compact Fourier dimension}, is thus defined by
	$$
	\dimfc A = 	\sup \left\{ s \in [0, d]; \, \widehat\mu(\xi) \lesssim |\xi|^{-s/2}, \,
	\mu \in \probs(K), \, K \subset A \text{ is compact} \right\}.
	$$
The anonymous referee of this paper provided an argument showing
that the compact Fourier dimension is countably stable whenever all
the sets in the union are closed (see Proposition~\ref{dimfcprop} below),
and pointed out that this can be used to deduce that the Fourier dimension
and the compact Fourier dimension are not the same. Inspired by this,
we then found an example that shows that the compact Fourier dimension
is not in general finitely stable.

The Hausdorff dimension is \emph{inner regular} in the
sense that
	$$
	\dimh A = \sup_{\substack{K \subset A \\ K \text{ compact}}} \dimh K
	$$
for any Borel set $A \subset \rea^d$ (this follows from \cite[Theorem 48]{rogers}),
and the same is true of the modified Fourier dimension by inner regularity of
finite Borel measures on $\rea^d$. Another way of expressing the fact that
\dimfc is different from \dimf is to say that \dimf is not inner regular.

One might consider to define the Fourier dimension and the modified Fourier
dimension of any $B \subset \rea^d$, by taking the supremum over all measures
in $\probs(\rea^d)$ that give full or positive measure to some Borel set $A \subset B$,
but then \dimf and \dimfm are not even finitely stable. For there is a construction
by Bernstein (using the well ordering theorem for sets with cardinality $\mathfrak{c}$)
of a set $B \subset \rea$ such that any closed subset of $B$ or $B^c$ is countable
\cite[Theorem~5.3]{oxtoby}. Thus any non-atomic measure $\mu \in \probs(\rea)$ gives measure
$0$ to any compact subset of $B$ or $B^c$, and by inner regularity to
any Borel subset of $B$ or $B^c$. It follows that $B$ and $B^c$ have Fourier
dimension and modified Fourier dimension $0$, but $B \cup B^c = \rea$ has dimension
$1$.

This can be modified slightly to produce Lebesgue measurable sets
$C_1, C_2 \subset \rea$ that would violate the finite stability. For
each natural number $n$, let $A_n$ be a Salem set of dimension $1 - 1 / n$ and
let
	$$
	C_1 = B \cap \bigcup_{n = 1}^\infty A_n, \qquad
	C_2 = B^c \cap \bigcup_{n = 1}^\infty A_n.
	$$
Then $C_1$ and $C_2$ are Lebesgue measurable since each $A_n$ has Lebesgue measure $0$,
and since they are subsets of $B$ and $B^c$ respectively they would have
Fourier dimension and modified Fourier dimension $0$. On the other hand,
	$$
	\dimf (C_1 \cup C_2) = \dimf\left(\bigcup_{n = 1}^\infty A_n\right) = 1,
	$$
and thus also $\dimfm(C_1 \cup C_2) = 1$.

\section{Stability of the Fourier dimension for sets}	\label{sec2}
In this section it is shown that the Fourier dimension is stable under
finite or countable unions of sets that satisfy a certain intersection condition,
and that the compact Fourier dimension is stable under finite or countable unions
of closed sets. Then examples are given, showing that the Fourier dimension is not
countably stable, that the compact Fourier dimension is different from the
Fourier dimension and that the compact Fourier dimension is not
finitely stable.

The following lemma is used in the proofs of Theorem~\ref{sepstabthm} and
Proposition~\ref{dimfcprop}, and also in Section~\ref{sec3} and Section~\ref{sec5}.

\begin{lemma}	\label{smoothcutlemma}
Let $\mu$ be a finite Borel measure on $\rea^d$ and let $f$ be a non-negative
$C^m$-function with compact support, where $m = \ceil{3d / 2}$. Define the
measure $\nu$ on $\rea^d$ by $\intd{\nu} = f \intd{\mu}$. Then
	$$
	\widehat\mu(\xi) \lesssim |\xi|^{-s/2}
	\implies
	\widehat\nu(\xi) \lesssim |\xi|^{-s/2}
	$$
for all $s \in [0, d]$, and in particular
	$$
	\dimf \nu \geq \dimf \mu.
	$$

\begin{proof}
Since $f$ is of class $C^m$ and has compact support, there is a constant $M$
such that
	$$
	|\widehat f (t)| \leq \frac{M}{1 + |t|^m}
	$$
for all $t \in \rea^d$. In particular $\widehat f$ is Lebesgue integrable,
and from this it also follows that the Fourier inversion formula holds
pointwise everywhere for $f$. Thus
	\begin{align*}
	\widehat \nu(\xi) &=
	\int e^{-2\pi i \xi \cdot x} f(x) \intd{\mu}(x) =
	\int e^{-2\pi i \xi \cdot x}
	\left( \int e^{2 \pi i t \cdot x} \widehat f(t) \intd{t} \right)\intd{\mu}(x) \\
	&=
	\int \left( \int e^{-2\pi i (\xi - t) \cdot x} \intd{\mu}(x) \right) \widehat f (t) \intd{t} =
	\int \widehat \mu (\xi - t) \widehat f(t) \intd{t}.
	\end{align*}
Now,
	\begin{align*}
	&\int_{\{|\xi - t| < |\xi| / 2\}} \left|\widehat \mu (\xi - t) \, \widehat f(t) \right| \intd{t}
	\leq
	\int_{\{ |t| \geq |\xi| / 2\}} \frac{\mu(\rea^d) \, M}{1 + |t|^m} \intd{t}
	\leq \\
	&\mu(\rea^d) \, M \int_{\{ |t| \geq |\xi| / 2\}} |t|^{-m} \intd{t}
	= \text{const.} \cdot {|\xi|}^{d - m} \lesssim |\xi|^{-d/2},
	\end{align*}
and if $\widehat \mu (\xi) \leq C |\xi|^{-s / 2}$ for all $\xi \in \rea^d$ then
	$$
	\int_{\{ |\xi - t | \geq |\xi| / 2 \}} 
	\left|\widehat \mu (\xi - t) \, \widehat f(t) \right| \intd{t}
	\leq
	\frac{C2^{s / 2}}{|\xi|^{s / 2}} \int |\widehat f (t)| \intd{t}
	\lesssim |\xi|^{-s / 2}.
	$$
Thus
	$$
	\widehat \nu(\xi) \lesssim |\xi|^{-d/2} + |\xi|^{-s / 2}
	$$
whenever $\widehat\mu(\xi) \lesssim |\xi|^{-s / 2}$, which proves the lemma.
\end{proof}
\end{lemma}

\begin{theorem}	\label{sepstabthm}
Let $\{ A_k \}$ be a finite or countable family of Borel subsets of\/ $\rea^d$
such that 
	$$
	\sup_n \, \dimfm \left( A_n \cap \closure{\bigcup_{k \neq n} A_k} \right) <
	\dimf \left( \bigcup_k A_k \right).
	$$
Then
	$$
	\dimf \left( \bigcup_k A_k \right) = \sup_k \, \dimf A_k.
	$$

\begin{proof}
The inequality $\geq$ is immediate from the monotonicity of $\dimf$.
To see the other inequality, take an arbitrary $\mu \in \probs(\bigcup_k A_k)$
such that
	$$
	\sup_n \, \dimfm \left( A_n \cap \closure{\bigcup_{k \neq n} A_k} \right)
	< \dimf \mu
	$$
and let $n$ be such that $\mu(A_n) > 0$. Then by the definition of the modified
Fourier dimension,
	$$
	\mu \left( A_n \cap \closure{\bigcup_{k \neq n} A_k} \right) = 0
	\quad
	\text{and thus}
	\quad
	\mu\left( A_n \setminus \closure{\bigcup_{k \neq n} A_k} \right) > 0.
	$$

Let $f$ be a non-negative $C^{\infty}$-function that is $0$ on $\closure{\bigcup_{k \neq n} A_k}$
and positive everywhere else. Then $\mu(f) > 0$, and there is a non-negative
$C^\infty$-function $g$ with compact support such that $\mu(fg) > 0$ as well. The
measure $\nu$ defined by
	$$
	\intd{\nu} = \frac{fg}{\mu(fg)} \intd{\mu}
	$$
is a probability measure on $A_n$, so
	$$
	\sup_k \, \dimf A_k \geq \dimf A_n \geq \dimf \nu \geq \dimf \mu,
	$$
where the last inequality is by Lemma~\ref{smoothcutlemma}. Taking supremum
on the right over all $\mu \in \probs(\bigcup_k A_k)$ gives the inequality
$\leq$ in the statement.
\end{proof}
\end{theorem}

\begin{corollary}	\label{thefirstcor}
Let $\{ A_k \}$ be a finite or countable family of Borel subsets of\/ $\rea^d$
such that 
	$$
	\sup_n \, \dimfm \left( A_n \cap \closure{\bigcup_{k \neq n} A_k} \right)
	\leq
	\sup_k \, \dimf A_k.
	$$
Then
	$$
	\dimf \left( \bigcup_k A_k \right) = \sup_k \, \dimf A_k.
	$$

\begin{proof}
If the conclusion does not hold then neither does the assumption, since then
	$$
	\sup_k \, \dimf A_k < \dimf \left( \bigcup_k A_k \right) \leq 
	\sup_n \, \dimfm \left( A_n \cap \closure{\bigcup_{k \neq n} A_k} \right),
	$$
where the second inequality is by Theorem~\ref{sepstabthm}.
\end{proof}
\end{corollary}

\begin{corollary}
Let $\{ A_k \}$ be a finite or countable family of Borel subsets of\/ $\rea^d$
such that 
	$$
	A_n \cap \closure{\bigcup_{k \neq n} A_k}
	$$
is countable for all $n$.
Then
	$$
	\dimf \left( \bigcup_k A_k \right) = \sup_k \, \dimf A_k.
	$$

\begin{proof}
This follows from Corollary~\ref{thefirstcor} since any countable set has modified
Fourier dimension $0$.
\end{proof}
\end{corollary}

The proof of the following proposition was provided by the anonymous
referee, who also noted that the statement can be applied to the sets
constructed in Example~\ref{setex} below to show that the compact Fourier
dimension is different from the Fourier dimension. Another application is
Example~\ref{setexc} below, which shows that the compact Fourier dimension is not
finitely stable.

\begin{proposition}	\label{dimfcprop}
Let $\{ A_k \}$ be a finite or countable family of closed subsets of\/ $\rea^d$.
Then
	$$
	\dimfc \left( \bigcup_k A_k \right) = \sup_k \, \dimfc A_k.
	$$

\begin{proof}
The inequality $\geq$ holds since \dimfc is monotone. To see the other
inequality, let $\mu$ be any Borel probability measure whose topological support
$K$ is a subset of $\bigcup_k A_k$. By Baire's category theorem the
complete metric space $K$ cannot be expressed as a countable union of 
nowhere dense sets, and since
	$$
	K = \bigcup_k K \cap A_k
	$$
it follows that there must be some $n$ such that $K \cap A_n = \closure{K \cap A_n}$
has non-empty interior in $K$. Thus there is an open subset $V$ of $\rea^d$ such that
	$$
	\emptyset \neq K \cap V \subset K \cap A_n.
	$$
Let $f$ be a non-negative $C^\infty$-function that is $0$ outside
of $V$ and satisfies $\mu(f) > 0$, and define the probability measure $\nu$ by
	$$
	\intd\nu = \frac{f}{\mu(f)} \intd\mu.
	$$
Then
	$$
	\dimf \mu \leq \dimf \nu \leq \dimfc A_n \leq \sup_k \, \dimfc A_k,
	$$
where the first inequality is by Lemma~\ref{smoothcutlemma} and second inequality
holds since $\nu$ gives full measure to $A_n$. Taking the supremum over all probability
measures $\mu$ such that $\supp \mu \subset \bigcup_k A_k$ completes the proof.
\end{proof}
\end{proposition}

Example~\ref{setex} below shows that the Fourier dimension is not countably stable, and also that
the strict inequality in the assumption of Theorem~\ref{sepstabthm}
cannot be changed to a non-strict inequality. The following lemma is used in
Example~\ref{setex} and Example~\ref{setexc}.

\begin{lemma}	\label{effinfsuplemma}
For any $\varepsilon \in (0, 1]$,
	$$
	\inf_{\phantom{!}\mu\phantom{!}} \sup_{j \geq 1} \left| \widehat \mu(j) \right|
	\geq \frac{\pi\varepsilon}{8 + 2\pi\varepsilon}
	\qquad\left( \, \geq \frac{\varepsilon}{5} \, \right),
	$$
where the infimum is over all $\mu \in \probs([\varepsilon, 1])$ and
the supremum is over all positive integers $j$.

\begin{proof}
Fix $\varepsilon > 0$ and take any $\mu \in \probs([\varepsilon, 1])$. If $\varphi$
is a real-valued continuous function supported on $[0, \varepsilon]$ such
that
	$$
	\int \varphi(x) \intd{x} = 1 \qquad \text{ and } \qquad
	\sum_{k = -\infty}^\infty |\widehat\varphi(k)| < \infty
	$$
%
then
	$$
	0 = \mu(\varphi) = \sum_{k = -\infty}^\infty \widehat\varphi(k) \overline{\widehat\mu(k)}
	= 1 + 2 \re\left( \sum_{k = 1}^\infty \widehat\varphi(k) \overline{\widehat\mu(k)} \right),
	$$
and thus
	$$
	\frac{1}{2} \leq \sum_{k = 1}^\infty |\widehat\varphi(k)| |\widehat\mu(k)|
	\leq \left(\sum_{k = 1}^\infty |\widehat\varphi(k)| \right) \left(\sup_{j \geq 1} |\widehat\mu(j)| \right).
	$$

Now let $\chi$ be the indicator function of $[0, \varepsilon / 2]$
and take $\varphi$ to be the triangle pulse
	$$
	\varphi(x) = \left( \frac{2\chi}{\varepsilon} \right) * 
	\left( \frac{2\chi}{\varepsilon} \right),
	$$
where $*$ denotes convolution. Then
	$$
	|\widehat\varphi(k)| =
	\left| \frac{2\widehat\chi(k)}{\varepsilon}\right|^2 =
	\sinc^2\left(\frac{k\pi\varepsilon}{2} \right) \leq
	\min\left( 1, \, \frac{4}{k^2 \pi^2 \varepsilon^2} \right),
	$$
so that
	$$
	\sum_{k = 1}^\infty |\widehat\varphi(k)| \leq
	\left\lceil \frac{2}{\pi\varepsilon} \right\rceil +
	\frac{4}{\pi^2 \varepsilon^2}
	\sum_{\ceil{\frac{2}{\pi\varepsilon}} + 1}^\infty \frac{1}{k^2}
	\leq
	\frac{2}{\pi\varepsilon} + 1 +
	\frac{4}{\pi^2 \varepsilon^2} 
	\int_{\frac{2}{\pi\varepsilon}}^\infty \frac{1}{x^2} \intd{x}
	= \frac{4 + \pi\varepsilon}{\pi\varepsilon}.
	$$
It follows that
	$$
	\sup_{j \geq 1} |\widehat\mu(j)| \geq \frac{1}{2} \cdot \frac{\pi\varepsilon}{4 + \pi\varepsilon}.
	$$
\end{proof}
\end{lemma}

\begin{example}	\label{setex}
Let $(l_k)_{k = 1}^\infty$ be a strictly increasing sequence of natural numbers such that
	$$
	\lim_{k \to \infty} \frac{l_k}{k} = \infty.
	$$
Define the compact sets
	$$
	A_k = \closure{\{ x \in [0, 1]; \, x_{l_k + 1} \ldots x_{l_k + k} \neq 0^k \}},
	$$
where $x = 0.x_1x_2\ldots$ is the binary decimal expansion of $x$, and let
	$$
	B_n = \bigcap_{k = n}^\infty A_k.
	$$

Take any measure $\mu \in \probs(B_n)$ and let $\mu_k$ be the image
of $\mu$ under the map $x \mapsto 2^{l_k} x \pmod{1}$. If $k \geq n$ then
$\mu_k$ gives full measure to $[2^{-k}, 1]$, so by Lemma \ref{effinfsuplemma} there is
some $j_k \geq 1$ such that
	$$
	\widehat\mu(2^{l_k} j_k) = \widehat\mu_k(j_k) \geq \frac{2^{-k}}{5}.
	$$
Thus for any $s > 0$,
	$$
	\limsup_{\xi \to \infty} |\widehat\mu(\xi)| |\xi|^{s / 2} \geq
	\lim_{k \to \infty} |\widehat\mu(2^{l_k} j_k)| \left( 2^{l_k} j_k \right)^{s / 2} \geq
	\lim_{k \to \infty} \frac{2^{s l_k / 2 - k}}{5} = \infty.
	$$
It follows that $\dimf(B_n) = 0$ for all $n$.

Let $\lambda$ be Lebesgue measure on $[0, 1]$. Then for each $n$
	$$
	\lambda\left(\bigcup_{n = 1}^\infty B_n \right) \geq 
	\lambda\left( B_n \right) \geq 1 - \sum_{k = n}^\infty 2^{-k} = 1 - 2^{-(n - 1)},
	$$
so that $\lambda\left(\bigcup_{n = 1}^\infty B_n \right) = 1$ and hence
$\dimf\left(\bigcup_{n = 1}^\infty B_n \right) = 1$. This shows that
the Fourier dimension is not countably stable, and that \dimfc is
different from \dimf since $\dimfc\left( \bigcup_{n = 1}^\infty B_n \right) = 0$
by Proposition~\ref{dimfcprop}.
\end{example}

For $d \geq 2$, the sets $\left\{ B_n \times [0, 1]^{d - 1}\right\}$ give a counterexample
to countable stability of the Fourier dimension in $\rea^d$. 
Moreover, these sets do not satisfy the conclusion of Theorem~\ref{sepstabthm}, but they
would satisfy the assumption if the strict inequality was replaced by a non-strict inequality.
Thus it is not possible to weaken the assumption of Theorem~\ref{sepstabthm} in that way.

The following variation of Example \ref{setex} shows that the compact Fourier dimension
is not finitely stable.

\begin{example}	\label{setexc}
Let $s \in (\sqrt{3} - 1, 1)$ and choose $b$ such that
	$$
	\frac{1 - s}{s} < b < \frac{s}{2}
	$$
(this is possible since $s > \sqrt{3} - 1$). Let $(l_k)_{k = 1}^\infty$ be
a strictly increasing sequence of natural numbers and let $m_k = \ceil{bl_k}$.
Define the compact sets
	$$
	A_k = \closure{\{ x \in [0, 1]; \, x_{l_k + 1} \ldots x_{l_k + m_k} \neq 0^{m_k} \}},
	$$
where $x = 0.x_1x_2\ldots$ is the binary decimal expansion of $x$, and let
	$$
	B_n = \bigcap_{k = n}^\infty A_k.
	$$

Take any measure $\mu \in \probs(B_n)$ and let $\mu_k$ be the image
of $\mu$ under the map $x \mapsto 2^{l_k} x \pmod{1}$. If $k \geq n$ then
$\mu_k$ gives full measure to $[2^{-m_k}, 1]$, so by Lemma \ref{effinfsuplemma} there is some
$j_k \geq 1$ such that
	$$
	\widehat\mu(2^{l_k} j_k) = \widehat\mu_k(j_k) \geq \frac{2^{-m_k}}{5}.
	$$
Thus
	$$
	\limsup_{\xi \to \infty} |\widehat\mu(\xi)| |\xi|^{s / 2} \geq
	\lim_{k \to \infty} |\widehat\mu(2^{l_k} j_k)| \left( 2^{l_k} j_k \right)^{s / 2} \geq
	\lim_{k \to \infty} \frac{2^{sl_k / 2 - m_k}}{5} = \infty,
	$$
and it follows that $\dimf(B_n) \leq s$ for all $n$. Consequently
$\dimfc \left(\bigcup_{n = 1}^\infty B_n\right) \leq s$ by Proposition~\ref{dimfcprop}.

Next consider the complement $C$ of $\bigcup_{n = 1}^\infty B_n$ in $[0, 1]$, that is,
	$$
	C = [0, 1] \setminus \bigcup_{n = 1}^\infty B_n =
	\bigcap_{n = 1}^\infty \bigcup_{k = n}^\infty \left([0, 1] \setminus A_k\right).
	$$
For each $k$ the set $[0, 1] \setminus A_k$ is a union of $2^{l_k}$ intervals of length
$2^{-(l_k + m_k)}$, and thus
	$$
	\mathcal{H}_\delta^s ([0, 1] \setminus A_k) \leq 2^{l_k} \cdot 2^{-s(l_k + m_k)},
	$$
whenever $k$ is so large that $l_k + m_k \geq -\log_2 \delta$. Thus for all large enough
$n$,
	$$
	\mathcal{H}_\delta^s(C) \leq \mathcal{H}_\delta^s\left(
	\bigcup_{k = n}^\infty \left( [0, 1] \setminus A_k \right) \right)
	\leq \sum_{k = n}^\infty 2^{(1 - s)l_k -sm_k}.
	$$
The sum converges by the choice of $(m_k)$, and since $n$ can be taken arbitrarily large
it follows that
	$$
	\mathcal{H}_\delta^s(C) = 0.
	$$
for all $\delta > 0$. Thus
	$$
	\dimfc C \leq \dimh C \leq s,
	$$
so that $[0, 1]$ is the union of two sets with compact Fourier dimension
strictly less than $1$.
\end{example}

\section{Stability of the Fourier dimension for measures}	\label{sec3}
As mentioned in the introduction, finite stability of the Fourier dimension
for sets would follow if it could be shown that
	\begin{equation} \label{measstabeq}
	\dimf (\mu + \nu) = \min(\dimf \mu, \, \dimf \nu)
	\end{equation}
for all finite Borel measures $\mu$ and $\nu$. The inequality
$\geq$ always holds, but Example~\ref{measex} below shows that
strict inequality is possible. The following lemma is used in
that example.

\begin{lemma}	\label{oscintlemma}
Let $\alpha$ and $\beta$ be two distinct real numbers. Then
	$$
	\left|
	\int_0^1 e^{2\pi i \alpha x} \sin(2 \pi \beta x) \intd{x}
	\right|
	\leq
	\frac{1}{\big| |\alpha| - |\beta| \big|}.
	$$

\begin{proof}
For any $\gamma > 0$,
	\begin{align*}
	\left|
	\int_0^1 e^{2\pi i \gamma x} \intd{x}
	\right|
	&=
	\left|
	\int_0^{\frac{\floor{\gamma}}{\gamma}} e^{2\pi i \gamma x} \intd{x} +
	\int_{\frac{\floor{\gamma}}{\gamma}}^1 e^{2\pi i \gamma x} \intd{x}
	\right| \\
	&=
	\left|
	\int_{\frac{\floor{\gamma}}{\gamma}}^1 e^{2\pi i \gamma x} \intd{x}
	\right|
	\leq
	1 - \frac{\floor{\gamma}}{\gamma}
	\leq
	\frac{1}{\gamma}.
	\end{align*}
From this together with the identity
	$$
	e^{2\pi i \alpha x} \sin(2\pi \beta x) =
	\frac{i}{2} \left(
	e^{2\pi i (\alpha - \beta) x} - e^{2\pi i (\alpha + \beta) x}
	\right),
	$$
it follows that
	\begin{align*}
	\left|
	\int_0^1 e^{2 \pi i \alpha x} \sin(2 \pi \beta x) \intd{x}
	\right|
	&\leq
	\frac{1}{2}
	\left|
	\int_0^1 e^{2 \pi i |\alpha - \beta| x} \intd{x}
	\right|
	+
	\frac{1}{2}
	\left|
	\int_0^1 e^{2 \pi i |\alpha + \beta| x} \intd{x}
	\right| \\
	&\leq
	\frac{1}{2} \left(
	\frac{1}{|\alpha - \beta|} + \frac{1}{|\alpha + \beta|}
	\right)
	\leq
	\frac{1}{\big| |\alpha| - |\beta| \big|}.
	\qedhere
	\end{align*}
\end{proof}
\end{lemma}

\begin{example}	\label{measex}
This example shows that the Fourier dimension for measures is not in general finitely stable.

Let
	$$
	g(x) = 1 + \sum_{k = 1}^\infty 2^{-k} \sin\left(2\pi \cdot 2^{k^2} x\right);
	$$
this is a continuous non-negative function. Define the probability measure
$\mu$ on $[0, 1]$ by $\intd{\mu} = g \intd{x}$.
Using that
	$$
	\int_0^1 e^{-2 \pi i l x} \sin(2 \pi l x) \intd{x} =
	\int_0^1 \cos(2 \pi l x) \sin(2 \pi l x) \intd{x} -
	i \int_0^1 \sin^2(2 \pi l x) \intd{x}
	= \frac{-i}{2}
	$$
for $l \in \nat$, one sees that for $n \geq 1$
	\begin{align*}
	\left|
	\widehat\mu\left(2^{n^2}\right) + i \cdot 2^{-(n + 1)}
	\right|
	&\leq
	\sum_{\substack{k \geq 1 \\ k \neq n}} 2^{-k}
	\left|
	\int_0^1 e^{-2 \pi i \cdot 2^{n^2} x} \sin(2 \pi \cdot 2^{k^2} x) \intd{x}
	\right| \\
	&\stackrel{*}{\leq}
	\sum_{\substack{k \geq 1 \\ k \neq n}} \frac{2^{-k}}{\left|2^{n^2} - 2^{k^2}\right|}
	\leq
	\frac{\sum_{k = 1}^\infty 2^{-k}}{2^{n^2} - 2^{(n - 1)^2}}
	\leq \frac{2}{2^{n^2}},
	\end{align*}
where the inequality at $*$ is by Lemma~\ref{oscintlemma}. Thus for any $s > 0$,
	$$
	\limsup_{|\xi| \to \infty} | \widehat\mu(\xi) | |\xi|^{s / 2} \geq
	\limsup_{n \to \infty} \left| \widehat\mu\left(2^{n^2}\right) \right| \cdot 2^{sn^2 / 2} \geq
	\lim_{n \to \infty} \left( 2^{-(n + 1)} - \frac{2}{2^{n^2}} \right) \cdot 2^{sn^2 / 2} = \infty,
	$$
and it follows that $\dimf \mu = 0$.

Next, let
	$$
	h(x) = 1 - \sum_{k = 1}^\infty 2^{-k} \sin(2\pi \cdot 2^{k^2} x)
	$$
and define the probability measure $\nu$ on $[0, 1]$ by $\intd{\nu} = h \intd{x}$.
Then $\dimf \nu = 0$ as well, but $\mu + \nu$ is twice Lebesgue measure, which
has Fourier dimension $1$.
\end{example}

Even though \eqref{measstabeq} does not hold in general, it does hold if $\mu$ and
$\nu$ have different Fourier dimensions. For suppose that, say,
$\dimf \mu < \dimf \nu$. Then for every $s \in (\dimf \mu, \, \dimf \nu)$ there
is a sequence $(\xi_k)$ with $|\xi_k| \to \infty$ such that
	$$
	\lim_{k \to \infty} |\widehat\mu(\xi_k)| |\xi_k|^{s/2} = \infty
	\qquad \text{and} \qquad
	\lim_{k \to \infty} |\widehat\nu(\xi_k)| |\xi_k|^{s/2} = 0,
	$$
so that
	$$
	\limsup_{|\xi| \to \infty}
	\left| \widehat\mu(\xi) + \widehat\nu(\xi) \right| |\xi|^{s / 2} = \infty
	$$
and hence $\dimf (\mu + \nu) \leq s$. Thus
	$$
	\dimf(\mu + \nu) \leq \inf \{ s \in (\dimf \mu, \, \dimf \nu) \} = \dimf \mu
	\leq \dimf(\mu + \nu).
	$$
For the same reason, any convex
combination of $\mu$ and $\nu$ satisfies
	$$
	\dimf((1 - \lambda) \mu + \lambda \nu) = \min(\dimf \mu, \, \dimf \nu).
	$$

Next suppose that $\dimf \mu = \dimf \nu = s$
and that there is some $\lambda_0 \in [0, 1]$ such that
	$$
	\dimf((1 - \lambda_0) \mu + \lambda_0 \nu) > s.
	$$
Then for any $\lambda \in [0, 1] \setminus \{ \lambda_0 \}$, the measure
$(1 - \lambda) \mu + \lambda \nu$ is a convex combination of
$(1 - \lambda_0) \mu + \lambda_0 \nu$ and one of $\mu, \nu$, so it
has Fourier dimension $s$. Thus there is at most one convex combination of
$\mu$ and $\nu$ that has Fourier dimension greater than $s$.

The results in the rest of this section describe situations where \eqref{measstabeq}
holds, or where it fails for at most one value of some parameter.

\begin{proposition}	\label{sepmeasprop}
Let $\mu$ and $\nu$ be finite Borel measures on $\rea^d$ whose supports are compact
and disjoint. Then
	$$
	\dimf(\mu + \nu) = \min(\dimf \mu, \, \dimf \nu).
	$$

\begin{proof}
Let $f$ be a non-negative, smooth and compactly supported function that has the
value $1$ on $\supp \mu$ and the value $0$ on $\supp \nu$.
Then $f \intd{(\mu + \nu)} = \intd{\mu}$, so
	$$
	\dimf \mu \geq \dimf(\mu + \nu)
	$$
by Lemma~\ref{smoothcutlemma}. Similarly,
	$$
	\dimf \nu \geq \dimf(\mu + \nu),
	$$
and thus
	$$
	\dimf(\mu + \nu) \leq \min(\dimf \mu, \, \dimf \nu).
	$$
The proposition now follows since the opposite inequality always holds.
\end{proof}
\end{proposition}

\begin{proposition}
Let $\mu$ be a finite Borel measure on $\rea^d$ with compact support
and let $\mu_t$ be the translation of $\mu$ by $t \in \rea^d$. Then
	$$
	\dimf (\mu + \mu_t) = \dimf \mu.
	$$

\begin{proof}
Since the Fourier dimension is translation invariant,
	$$
	\dimf (\mu + \mu_t) \geq \min(\dimf \mu , \, \dimf \mu_t) = \dimf \mu.
	$$
The opposite inequality clearly holds if $t = 0$, so assume that $t \neq 0$ and
let $n$ be an odd integer so large that $\supp \mu \cap \supp \mu_{nt} = \emptyset$.
Note that
	$$
	\left| \widehat{\mu + \mu}_t(\xi) \right| =
	\left| \left( 1 + e^{-2\pi i t \cdot \xi } \right) \widehat\mu(\xi) \right|
	=
	2\, |\cos(\pi t \cdot \xi)| |\widehat\mu(\xi)|,
	$$
and similarly
	$$
	\left| \widehat{\mu + \mu}_{nt}(\xi) \right| =
	2\, |\cos(\pi n t \cdot \xi)| |\widehat\mu(\xi)|.
	$$
Since $\cos(nx) / \cos x$ is bounded, this gives
	$$
	\left| \widehat{\mu + \mu}_{nt}(\xi) \right| =
	2\, \left| \frac{\cos(\pi n t \cdot \xi)}{\cos(\pi t \cdot \xi)} \right|
	|\cos(\pi t \cdot \xi)||\widehat\mu(\xi)|
	\lesssim
	\left| \widehat{\mu + \mu}_t(\xi) \right|.
	$$
Thus
	$$
	\dimf (\mu + \mu_t) \leq \dimf (\mu + \mu_{nt}) =
	\min (\dimf \mu, \dimf \mu_{nt} ) = \dimf \mu,
	$$
where the first equality is by Proposition~\ref{sepmeasprop}.
\end{proof}
\end{proposition}

\begin{proposition}
Let $\mu$ and $\nu$ be finite Borel measures on $\rea^d$ with compact supports,
and for $t \in \rea^d$ let $\nu_t$ be the translation of $\nu$ by $t$. Then there is
at most one $t$ such that
	$$
	\dimf(\mu + \nu_t) > \min(\dimf \mu, \, \dimf \nu).
	$$

\begin{proof}
It shall be shown that
	$$
	\min(\dimf(\mu + \nu_{t_1}), \, \dimf(\mu + \nu_{t_2})) \leq
	\min(\dimf \mu, \, \dimf \nu)
	$$
whenever $t_1 \neq t_2$. By the translation invariance of $\dimf$, this
is equivalent to
	$$
	\min(\dimf(\kappa + \nu), \, \dimf(\kappa + \nu_\Delta)) \leq
	\min(\dimf \kappa, \, \dimf \nu),
	$$
where $\kappa$ is the translation of $\mu$ by $-t_1$ and $\Delta = t_2 - t_1$.
Suppose that $s$ is less than the expression on the left --- then $s < d$ and
	\begin{equation} \label{relsyst}
	\left\{
	\begin{array}{l}
	\widehat\kappa(\xi) + \phantom{e^{-2\pi i \Delta \cdot \xi} \, }\widehat\nu(\xi) \lesssim |\xi|^{-s/2} \\
	\widehat\kappa(\xi) + e^{-2\pi i \Delta \cdot \xi} \, \widehat\nu(\xi) \lesssim |\xi|^{-s/2}.
	\end{array}
	\right.
	\end{equation}
Let $n$ be an integer so large that $\supp \kappa \cap \supp \nu_{n \Delta} = \emptyset$.
Subtracting the second relation in \eqref{relsyst} from the first gives
	$$
	\sin(\pi\Delta \cdot \xi) \widehat\nu(\xi)
	\lesssim
	\left( 1 - e^{-2\pi i \Delta \cdot \xi} \right) \widehat\nu(\xi)
	\lesssim
	|\xi|^{-s / 2},
	$$
and since $\sin(nx)/\sin x$ is bounded it follows that
	$$
	\left( 1 - e^{-2\pi i n \Delta \cdot \xi} \right) \widehat\nu(\xi)
	\lesssim
	\frac{\sin(\pi n \Delta \cdot \xi)}{\sin(\pi\Delta \cdot \xi)}
	\sin(\pi\Delta \cdot \xi) \widehat\nu(\xi) \lesssim |\xi|^{-s / 2}.
	$$
Subtracting this from the first relation in \eqref{relsyst} then gives
	$$
	\widehat\kappa(\xi) + e^{-2\pi i n\Delta \cdot \xi} \, \widehat\nu(\xi) \lesssim |\xi|^{-s / 2},
	$$
and thus
	$$
	s \leq \dimf (\kappa + \nu_{n\Delta}) = \min(\dimf \kappa, \, \dimf \nu_{n\Delta})
	= \min(\dimf \kappa, \dimf \nu),
	$$
where the first equality is by Proposition~\ref{sepmeasprop}.
\end{proof}
\end{proposition}

\begin{lemma}	\label{Blemma}
Let $B \in \rea^{d \times d}$ be an invertible matrix such that $|\lambda| \neq 1$ for 
all eigenvalues $\lambda$ of $B$, and let $f: \rea^d \to \rea$ be a function such
that $\lim_{|\xi| \to \infty} f(\xi) = 0$ and
	$$
	\left| f(\xi) - f(B\xi) \right| \lesssim |\xi|^{-\alpha}
	$$
for some $\alpha > 0$. Then
	$$
	f(\xi) \lesssim |\xi|^{-\alpha}.
	$$

\begin{proof}
Using that $|B^{-1} \xi| \geq \| B \|^{-1} |\xi|$ one sees that
	$$
	\left| f(B^{-1}\xi) - f(\xi) \right| \lesssim |B^{-1}\xi|^{-\alpha} \lesssim |\xi|^{-\alpha},
	$$
and thus there is a constant $C$ such that
	$$
	|f(\xi) - f(B\xi)| \leq C |\xi|^{-\alpha}
	\quad\text{and}\quad
	|f(\xi) - f(B^{-1}\xi)|	\leq C |\xi|^{-\alpha}
	$$
for all $\xi$.

Let
	$$
	V_u = \bigoplus_{|\lambda| > 1} E_\lambda,
	\qquad
	V_s = \bigoplus_{|\lambda| < 1} E_\lambda,	
	$$
where $E_\lambda$ denotes the (generalised) eigenspace of $B$ corresponding to the eigenvalue
$\lambda$. Then there is some $c > 1$ and an $m$ such that if $k \geq m$ then
$|B^k \xi| \geq c^k |\xi|$ for all $\xi \in V_u$ and $|B^{-k}\xi| \geq c^k |\xi|$ for all
$\xi \in V_s$ (this need not be true with $m = 1$, for instance if $B$ is
a large Jordan block with diagonal entries only slightly larger than $1$).

Take any $\xi \in \rea^d \setminus \{ 0 \}$ and write it as $\xi = \xi_u + \xi_s$ with 
$\xi_u \in V_u$ and $\xi_s \in V_s$. Suppose first that
$|\xi_u| \geq |\xi_s|$, or equivalently that $|\xi_u| \geq |\xi| / \sqrt{2}$.
Then for any $n \geq m$,
	\begin{align*}
	|f(\xi) - f(B^n\xi)| &\leq
	\sum_{k = 0}^{n - 1} |f(B^k\xi) - f(B^{k + 1}\xi)|
	\leq
	C \sum_{k = 0}^{n - 1} |B^k\xi|^{-\alpha}
	\leq
	C \sum_{k = 0}^{n - 1} |B^k\xi_u|^{-\alpha} \\
	&\leq
	C\left( \sum_{k = 0}^{m - 1} \big\| B^{-1} \big\|^{k \alpha} + 
	\sum_{k = m}^\infty c^{-k\alpha} \right) |\xi_u|^{-\alpha} \\
	&\leq
	C \left( \sum_{k = 0}^{m - 1} \big\| B^{-1} \big\|^{k \alpha} +
	\sum_{k = m}^\infty c^{-k\alpha} \right) 2^{\alpha / 2} \, |\xi|^{-\alpha}.
	\end{align*}
Letting $n \to \infty$ shows that $|f(\xi)| \leq D |\xi|^{-\alpha}$ with
	$$
	D = 2 C \left(
	\max\left(\sum_{k = 0}^{m - 1} \big\| B^{-1} \big\|^{k \alpha}, \, 
	\sum_{k = 0}^{m - 1} \big\| B      \big\|^{k \alpha}\right)
	+ \sum_{k = m}^\infty c^{-k\alpha} \right) 2^{\alpha / 2}.
	$$
Similarly, if $|\xi_s| \geq |\xi_u|$ then for any $n \geq m$,
	\begin{align*}
	|f(\xi) - f(B^{-n}\xi)| \big| \leq \ldots \leq
	C \left( \sum_{k = 0}^{m - 1} \big\| B \big\|^{k \alpha} +
	\sum_{k = m}^\infty c^{-k\alpha} \right) 
	2^{\alpha / 2} \, |\xi|^{-\alpha},
	\end{align*}
and letting $n \to \infty$ shows that $|f(\xi)| \leq D |\xi|^{-\alpha}$ in this case as well.
\end{proof}
\end{lemma}

\begin{proposition}
Let $\mu$ be a finite Borel measure on $\rea^d$ such that
$\lim_{|\xi| \to \infty} \widehat\mu(\xi) = 0$, and let $A \in \rea^{d \times d}$
be an invertible matrix such that $|\lambda| \neq 1$ for all eigenvalues $\lambda$ of $A$. Then
	$$
	\dimf (\mu + A \mu) = \dimf \mu.
	$$

\begin{proof}
Since the Fourier dimension is invariant under invertible linear
transformations,
	$$
	\dimf (\mu + A\mu) \geq \min(\dimf \mu, \, \dimf(A\mu)) = \dimf \mu,
	$$
and in particular the lemma is true if $\dimf (\mu + A\mu) = 0$.
To see the opposite inequality when $\dimf (\mu + A\mu) > 0$, 
take any $s \in (0, \dimf (\mu + A\mu))$ and let $B = A^T$, so that
	$$
	\widehat{A\mu}(\xi) = \widehat\mu(A^T \xi) = \widehat\mu(B\xi).
	$$
Then
	$$
	\big| |\widehat\mu(\xi)| - |\widehat\mu(B\xi)| \big|
	\leq
	|\widehat\mu(\xi) + \widehat\mu(B\xi)| \lesssim |\xi|^{- s / 2},
	$$
so Lemma~\ref{Blemma} applied to $f(\xi) = |\widehat\mu(\xi)|$ says that
$\widehat\mu(\xi) \lesssim |\xi|^{-s / 2}$ and therefore $s \leq \dimf \mu$.
\end{proof}
\end{proposition}

\begin{proposition}
Let $\mu$ and $\nu$ be finite Borel measures on $\rea^d$ such that
	$$
	\lim_{|\xi| \to \infty} \widehat\mu(\xi) = \lim_{|\xi| \to \infty} \widehat\nu(\xi) = 0,
	$$
and let $A \in \rea^{d \times d}$
be a matrix such that $\re \lambda \neq 0$ for all eigenvalues $\lambda$ of $A$.
For $t \in \rea$, let $\nu_t = \exp(tA) \nu$. Then there is at most one $t$ such
that
	$$
	\dimf(\mu + \nu_t) > \min(\dimf \mu, \, \dimf \nu).
	$$

\begin{proof}
The statement is trivially true if $\mu$ and  $\nu$ have different Fourier
dimensions, so assume that
$\dimf \mu = \dimf \nu$. Take any distinct $t_1, t_2$ and suppose that
	\begin{align*}
	s
	&< \min(\dimf(\mu + \nu_{t_1}), \, \dimf(\mu + \nu_{t_2})) \\
	&= \min(\dimf(\kappa + \nu), \, \dimf(\kappa + \nu_{t_2 - t_1})),
	\end{align*}
where $\kappa = \exp(-t_1A)\mu$. Then $s < d$ and
	$$
	\left\{
	\begin{array}{ll}
	\widehat\kappa(\xi) + \widehat\nu(\xi)		&\lesssim |\xi|^{-s/2} \\
	\widehat\kappa(\xi) + \widehat\nu(B \xi)	&\lesssim |\xi|^{-s/2},
	\end{array}
	\right.
	$$
where $B = \exp((t_2 - t_1) A)^T$,
and subtracting the second relation from the first gives
	$$
	\big| |\widehat\nu(\xi)| - |\widehat\nu(B\xi)| \big|
	\leq
	\left| \widehat\nu(\xi) - \widehat\nu( B \xi) \right| \lesssim |\xi|^{- s / 2}.
	$$
The matrix $B$ has no eigenvalue on the unit circle, so
$\widehat\nu(\xi) \lesssim |\xi|^{-s / 2}$ by Lemma~\ref{Blemma} applied to
$f(\xi) = |\widehat\nu(\xi)|$. Thus $s \leq \dimf \nu$, which concludes the proof.
\end{proof}
\end{proposition}

\section{The modified Fourier dimension}	\label{sec4}
Recall that the modified Fourier dimension of a Borel set $A \subset \rea^d$ is
defined by
	$$
	\dimfm A = \sup\left\{\dimf \mu; \, \mu \in \mathcal{P}(\rea^d), \,
	\mu(A) > 0 \right\}.
	$$

\begin{theorem}
The modified Fourier dimension is monotone and countably stable, and
satisfies $\dimf A \leq \dimfm A \leq \dimh A$ for any Borel set
$A \subset \rea^d$.

\begin{proof}
If $A \subset B$ then
$\{ \mu \in \probs(\rea^d); \, \mu(A) > 0 \} \subset \{ \mu \in \probs(\rea^d); \, \mu(B) > 0\}$, so
	\begin{align*}
	\dimfm A &=
	\sup \{ \dimf \mu; \, \mu \in \probs(\rea^d), \, \mu(A) > 0 \} \\ &\leq
	\sup \{ \dimf \mu; \, \mu \in \probs(\rea^d), \, \mu(B) > 0 \} =
	\dimfm B.
	\end{align*}
Thus $\dimfm$ is monotone.

Let $\{ A_k \}$ be a finite or countable family of Borel sets. For any
$\mu \in \probs(\rea^d)$ such that $\mu(\bigcup A_k) > 0$ there must
be some $n$ such that $\mu(A_n) > 0$, and thus
	$$
	\sup_k \, \dimfm A_k \geq \dimfm A_n \geq \dimf \mu.
	$$
Taking the supremum on the right over
$\{ \mu \in \probs(\rea^d); \, \mu(\bigcup A_k) > 0\}$ shows that
	$$
	\sup_k \, \dimfm A_k \geq \dimfm\left( \bigcup_k A_k \right).
	$$
The opposite inequality holds by monotonicity.

It is obvious that $\dimf A \leq \dimfm A$ since
any $\mu \in \probs(\rea^d)$ that gives full measure to $A$ in particular
gives positive measure to $A$. The proof outlined in the introduction of
the inequality $\dimf A \leq \dimh A$ works without modification if
$\dimf$ is replaced by $\dimfm$.
\end{proof}
\end{theorem}

The following two examples show that $\dimfm$ is not the same as either
of $\dimf$ and $\dimh$.

\begin{example}
The sets $B_n$ defined in Example~\ref{setex} were shown to have Fourier
dimension $0$ but positive Lebesgue measure, and hence modified Fourier
dimension $1$.
\end{example}

\begin{example} \label{cantorex}
The middle-third Cantor set has modified Fourier dimension $0$ (see the introduction),
but Hausdorff dimension $\log 2 / \log 3$.
\end{example}

\section{Null sets of $s$-dimensional measures}	\label{sec5}
Let
	$$
	\mathcal{M}_s = \left\{ \mu \in \probs(\rea^d); \, \dimfm \mu \geq s \right\}.
	$$
In this section it will be shown that $\mathcal{M}_s$ is characterised by its class
of common null sets, or more precisely that
	\begin{equation} \label{aimeq}
	\mu \in \mathcal{M}_s \, \iff \, \mu(E) = 0 \text{ for all } E \in \mathcal{E}_s,
	\end{equation}
where
	$$
	\mathcal{E}_s = \{ E \in \mathcal{B}(\rea^d); \,
	\mu(E) = 0 \text{ for all } \mu \in \mathcal{M}_s\}.
	$$
For $\mathcal{C} \subset \mathcal{P}(\rea^d)$ and $\mathcal{E} \subset \mathcal{B}(\rea^d)$
let
	\begin{align*}
	&\mathcal{C}^\perp = \{ E \in \mathcal{B}(\rea^d); \,
	\mu(E) = 0 \text{ for all } \mu \in \mathcal{C}\} \\
	&\mathcal{E}^\perp = \{ \mu \in \mathcal{P}(\rea^d); \,
	\mu(E) = 0 \text{ for all } E \in \mathcal{E}\}.
	\end{align*}
Then $\mathcal{E}_s = \mathcal{M}_s^\perp$ and the condition \eqref{aimeq} can be expressed
as $\mathcal{M}_s = \mathcal{E}_s^\perp$, or equivalently as $\mathcal{M}_s^{\perp\perp} = \mathcal{M}_s$.

It is also natural to consider the sets
	$$
	\mathcal{C}_s = \{ \mu \in \mathcal{P}(\rea^d); \, \exists \nu \in \mathcal{P}(\rea^d)
	\text{ such that } \mu \ll \nu \text{ and } \widehat\nu(\xi) \lesssim |\xi|^{-s / 2} \}.
	$$
For $s \in (0, d]$ they are related to $\mathcal{M}_s$ by
	$$
	\mathcal{M}_s = \bigcap_{t < s} \mathcal{C}_t
	$$
(this is also true for $s = 0$ if one allows negative $t$:s in the intersection).

The goal of this section is to prove the following theorem.

\begin{theorem}	\label{perpperpthm}
The sets $\mathcal{C}_s$ and $\mathcal{M}_s$ satisfy
	$$
	\begin{array}{rll}
	\mathcal{C}_s^{\perp\perp} \!\!\!\!\! &= \mathcal{C}_s, \quad &\text{for } 0 \leq s, \\
	\mathcal{M}_s^{\perp\perp} \!\!\!\!\! &= \mathcal{M}_s, \quad &\text{for } 0 \leq s \leq d.
	\end{array}
	$$
\end{theorem}

The theorem is trivial for $s = 0$ since
$\mathcal{C}_0 = \mathcal{M}_0 = \probs(\rea^d)$ and
$\mathcal{C}_0^\perp = \mathcal{M}_0^\perp = \{ \emptyset \}$. For general $s$, the
first step in the proof is to reduce the problem using some properties of $\perp$
that are collected in the next lemma.

\begin{lemma}	\label{perppropslemma}
Let $\mathcal{D}$, $\mathcal{D}_1$, $\mathcal{D}_2$ and
$\{ \mathcal{D}_\alpha \}_{\alpha \in I}$ be subsets of either $\probs(\rea^d)$
or $\mathcal{B}(\rea^d)$. Then
\begin{enumerate}[i)]
\item
$\mathcal{D} \subset \mathcal{D}^{\perp\perp}$

\item
$\mathcal{D}_1 \subset \mathcal{D}_2 \implies \mathcal{D}_2^\perp \subset \mathcal{D}_1^\perp$

\item
$\mathcal{D}^{\perp\perp\perp} = \mathcal{D}^\perp$

\item
$\bigcup_{\alpha \in I} \mathcal{D}_\alpha^\perp \subset
\left( \bigcap_{\alpha \in I} \mathcal{D}_\alpha \right)^\perp$

\item
$\bigcap_{\alpha \in I} \mathcal{D}_\alpha^\perp =
\left( \bigcup_{\alpha \in I} \mathcal{D}_\alpha \right)^\perp$.
\end{enumerate}
\end{lemma}

Property \emph{iii)} implies that $\mathcal{D}^{\perp(k + 2)} = \mathcal{D}^{\perp k}$
for any $k \geq 0$, as illustrated by the following diagram.
\begin{center}
\begin{tikzpicture}
\matrix (m) [
			matrix of math nodes,
			row sep=1em,
			column sep=1em,
			text height=1.5ex, text depth=0.25ex
			] {
			\phantom{^{\perp\perp\perp}}\mathcal{D}\phantom{\perp}
			& \subset
			& \phantom{^\perp}\mathcal{D}^{\perp\perp}\phantom{^\perp}
			& =
			& \mathcal{D}^{\perp\perp\perp\perp}
			& \cdots	\\
			& \phantom{^{\perp\perp}}\mathcal{D}^{\perp}\phantom{^\perp}
			& =
			& \phantom{^\perp}\mathcal{D}^{\perp\perp\perp}
			& =		
			& \cdots \\
		};

\path[>=stealth,->]
        (m-1-1) edge (m-2-2)
        (m-2-2) edge (m-1-3)
        (m-1-3) edge (m-2-4)
        (m-2-4) edge (m-1-5);
\end{tikzpicture}
\end{center}

\begin{proof}
Let $\mathcal{X}$ be the space that the $\mathcal{D}$:s are subsets of
and let $\mathcal{Y}$ be the ``dual'' space, that is,
	$$
	\mathcal{Y} =
	\begin{cases}
	\mathcal{B}(\rea^d)	&\text{if } \mathcal{X} = \probs(\rea^d) \\
	\probs(\rea^d) 		&\text{if } \mathcal{X} = \mathcal{B}(\rea^d).
	\end{cases}
	$$
If $x \in \mathcal{X}$ and $y \in \mathcal{Y}$, then the equation
$\{ x, y \} = \{ \mu, E\}$ determines uniquely a measure $\mu$ and
a set $E$. Thus it is possible to define
	$$
	(x, y) = \mu(E), \qquad \text{where } \{ x, y \} = \{ \mu, E \}. 
	$$

\emph{i) } Let $x \in \mathcal{D}$. Then $(x, y) = 0$ for any
$y \in \mathcal{D}^\perp$, which by definition of $\mathcal{D}^{\perp\perp}$
means that $x \in \mathcal{D}^{\perp\perp}$.

\emph{ii) } Let $y \in \mathcal{D}_2^\perp$. Then $(x, y) = 0$ for all
$x \in \mathcal{D}_2$, and thus $(x, y) = 0$ for all $x \in \mathcal{D}_1$
since $\mathcal{D}_1 \subset \mathcal{D}_2$. Hence $y \in \mathcal{D}_1^\perp$.

\emph{iii) } Applying \emph{ii)} to the statement of \emph{i)} shows that
$\mathcal{D}^{\perp\perp\perp} \subset \mathcal{D}^\perp$ and applying \emph{i)}
to $\mathcal{D^\perp}$ shows that $\mathcal{D}^\perp \subset \mathcal{D}^{\perp\perp\perp}$.

\emph{iv) } From \emph{ii)} it follows that
	$$
	\mathcal{D}_\alpha^\perp \subset \left( \bigcap_{\alpha' \in I} \mathcal{D}_{\alpha'} \right)^\perp
	$$
for any $\alpha \in I$, and hence
	$$
	\bigcup_{\alpha \in I} \mathcal{D}_\alpha^\perp \subset
	\left( \bigcap_{\alpha' \in I} \mathcal{D}_{\alpha'} \right)^\perp.
	$$

\emph{v) } The definition of $\mathcal{D}^\perp$ can be expressed as
	$$
	\mathcal{D}^\perp =
	\bigcap_{x \in \mathcal{D}} \{ y \in \mathcal{Y}; \, (x, y) = 0\} =
	\bigcap_{x \in \mathcal{D}} \{ x \}^\perp,
	$$
and thus
	\begin{align*}
	\left( \bigcup_{\alpha \in I} \mathcal{D}_\alpha \right)^\perp
	=
	\bigcap_{x \in \bigcup_{\alpha \in I} \mathcal{D}_\alpha} \{ x \}^\perp
	=
	\bigcap_{\alpha \in I} \bigcap_{x \in \mathcal{D}_\alpha} \{ x \}^\perp
	=
	\bigcap_{\alpha \in I} \mathcal{D}_\alpha^\perp.
	&\qedhere
	\end{align*}
\end{proof}

Once it has been proved that $\mathcal{C}_s^{\perp\perp} = \mathcal{C}_s$ it
follows by Lemma~\ref{perppropslemma} that $\mathcal{M}_s^{\perp\perp} = \mathcal{M}_s$
as well, for then
	$$
	\mathcal{M}_s^{\perp\perp} =
	\left( \bigcap_{t < s} \mathcal{C}_t \right)^{\perp\perp} \subset
	\left( \bigcup_{t < s} \mathcal{C}_t^\perp \right)^\perp =
	\bigcap_{t < s} \mathcal{C}_t^{\perp\perp} =
	\bigcap_{t < s} \mathcal{C}_t =
	\mathcal{M}_s
	$$
by \emph{ii)}, \emph{iv)} and \emph{v)}, and the opposite inclusion holds by \emph{i)}.
Moreover, if $\mathcal{D}$ has the form $\mathcal{D} = \mathcal{D}'^{\perp\perp}$ then
\emph{iii)} gives
	$$
	\mathcal{D}^{\perp\perp} = \mathcal{D}'^{\perp\perp\perp\perp}
	= \mathcal{D}'^{\perp\perp}
	= \mathcal{D}.
	$$
To prove Theorem~\ref{perpperpthm}, it thus suffices to show that
$\mathcal{C}_s = \mathcal{C}_s'^{\perp\perp}$, where
	$$
	\mathcal{C}_s' = \{ \nu \in \mathcal{P}(\rea^d); \, \widehat\nu(\xi) \lesssim |\xi|^{-s / 2} \}.
	$$
It is easy to see that $\mathcal{C}_s \subset \mathcal{C}_s'^{\perp\perp}$ (see the first part
of the proof of Theorem~\ref{perpperpthm}, on page~\pageref{perpperpproof}), but the other inclusion
takes a bit of work. The idea is to take an arbitrary measure $\mu \in \mathcal{C}_s'^{\perp\perp}$
and decompose it as $\mu = \mu_1 + \mu_2$ such that $\mu_1$ is absolutely continuous
to some measure in $\mathcal{C}_s'$ (thus $\mu_1 \in \mathcal{C}_s$) and $\mu_2$ is singular
to all measures in $\mathcal{C}_s'$, and then show that $\mu_2 = 0$ so that
$\mu = \mu_1 \in \mathcal{C}_s$.

\subsection{Decomposition of $\mu$ with respect to $\mathcal{C}'_s$}
\begin{definition}
A set $\mathcal{C} \subset \mathcal{P}(\rea^d)$ is
\emph{countably quasiconvex} if for any finite
or infinite sequence $(\nu_k)$ in $\mathcal{C}$ there is a sequence $(p_k)$
of positive numbers such that $\sum_k p_k = 1$ and
	$$
	\sum_k p_k \nu_k \in \mathcal{C}.
	$$
\end{definition}

Thus any countable convex combination of measures in a countably quasiconvex
set $\mathcal{C}$ is equivalent to some measure in $\mathcal{C}$.

\begin{lemma}	\label{Ciscpsclemma}
For any $s \in \rea$, the set $\mathcal{C}_s'$ is countably quasiconvex.

\begin{proof}
If $\nu_1, \nu_2, \ldots \in \mathcal{C}_s'$, then there are constants
$C_1, C_2, \ldots \geq 1$ such that for each $k$
	$$
	|\widehat\nu_k(\xi)| \leq C_k |\xi|^{-s / 2}
	\qquad\text{ for all } \xi.
	$$
Now set
	$$
	a_k = \frac{1}{2^kC_k},
	\qquad\qquad
	p_k = \frac{a_k}{\sum_{i = 1}^\infty a_i}.
	$$
Then the probability measure
	$$
	\nu = \sum_{k = 1}^\infty p_k \nu_k
	$$
satisfies
	$$
	|\widehat\nu(\xi)| \leq \sum_{k = 1}^\infty p_k|\widehat\nu_k(\xi)| \leq
	\sum_{k = 1}^\infty p_kC_k|\xi|^{-s / 2} \lesssim |\xi|^{-s / 2},
	$$
so $\nu \in \mathcal{C}_s'$.
\end{proof}
\end{lemma}

\begin{lemma}	\label{decomplemma}
Let $\mathcal{C} \subset \mathcal{P}(\rea^d)$ be countably quasiconvex and
let $\mu \in \mathcal{P}(\rea^d)$. Then there is a set $E \in \mathcal{B}(\rea^d)$
such that
	\begin{align*}
	&\restr{\mu}{E} \ll \nu \quad \text{ for some } \nu \in \mathcal{C}, \qquad\text{ and} \\
	&\restr{\mu}{E^c} \perp \nu \quad \text{ for all } \nu \in \mathcal{C}.
	\end{align*}

\begin{proof}
Let
	$$
	r = \sup \left\{ \mu(F); \, F \in \mathcal{B}(\rea^d) \text{ and }
	\restr{\mu}{F} \ll \nu \text{ for some } \nu \in \mathcal{C} \right\},
	$$
and for $k = 1, 2, \ldots$ let $\nu_k \in \mathcal{C}$ and $F_k \in \mathcal{B}(\rea^d)$
be such that $\restr{\mu}{F_k} \ll \nu_k$ and $\mu(F_k) \geq r - 1 / k$. Set
	$$
	E = \bigcup_{k = 1}^\infty F_k.
	$$

By assumption, there is a sequence $(p_k)$ of positive numbers such that
	$$
	\sum_{k = 1}^\infty p_k \nu_k \in \mathcal{C},
	$$
and since all $p_k$ are positive, $\restr{\mu}{E}$ is absolutely continuous with respect
to this measure.

Suppose towards a contradiction that there is some $\nu \in \mathcal{C}$
such that $\restr{\mu}{E^c} \not \perp \nu$. Then by Lebesgue decomposition
of $\restr{\mu}{E^c}$ with respect to $\nu$ there is a Borel set $S \subset E^c$
such that
	$$
	\restr{\mu}{S} \ll \nu \quad \text{and} \quad \mu(S) > 0.
	$$
For each $k$, there is a $\lambda_k \in (0, 1)$ such that
	$$
	(1 - \lambda_k) \nu + \lambda_k \nu_k \in \mathcal{C},
	$$
and $\restr{\mu}{S \cup F_k}$ is absolutely continuous with respect to
this measure. Moreover,
	$$
	\mu(S \cup F_k) = \mu(S) + \mu(F_k) \geq \mu(S) + r - \frac{1}{k}
	$$
which is greater than $r$ for large enough $k$ --- this is a contradiction.
\end{proof}
\end{lemma}

\subsection{Proof of Theorem~\ref{perpperpthm}}
The following theorem by Goullet de Rugy \cite{goulletderugy} is used in the proof.

\begin{theorem}[Goullet de Rugy]	\label{septhm}
Let $T$ be a compact Hausdorff space and let $\mathcal{A}$ and $\mathcal{B}$
be subsets of $\probs(T)$ of the form
	$$
	\mathcal{A} = \bigcup_{k = 1}^\infty \mathcal{A}_k, \qquad
	\mathcal{B} = \bigcup_{k = 1}^\infty \mathcal{B}_k,
	$$
where the $\mathcal{A}_k$ and $\mathcal{B}_k$ are weak-$*$ compact
and convex, such that $\mu \perp \nu$ for all $\mu \in \mathcal{A}$,
$\nu \in \mathcal{B}$. Then there exist disjoint $F_{\sigma\delta}$-sets
$T_1, T_2 \subset T$ such that $\mu(T_1) = 1$ for all $\mu \in \mathcal{A}$
and $\nu(T_2) = 1$ for all $\nu \in \mathcal{B}$.
\end{theorem}

\begin{proof}[Proof of Theorem~\ref{perpperpthm}] \label{perpperpproof}
As remarked in the beginning of the section, it suffices to show that
$\mathcal{C}_s'^{\perp\perp} = \mathcal{C}_s$. For any $\mu \in \mathcal{C}_s$
there is some $\nu \in \mathcal{C}_s'$ such that $\mu \ll \nu$. Then any
$E \in \mathcal{C}_s'^\perp$ is a null set for $\nu$ and hence also for
$\mu$, which means that $\mu \in \mathcal{C}_s'^{\perp\perp}$. Thus
$\mathcal{C}_s \subset \mathcal{C}_s'^{\perp\perp}$.

To see the other inclusion, take any $\mu \in \mathcal{C}_s'^{\perp\perp}$.
By Lemma~\ref{Ciscpsclemma} and Lemma~\ref{decomplemma}, it is possible to write
$\mu = \mu_1 + \mu_2$ such that $\mu_1$ is absolutely continuous to some
measure in $\mathcal{C}_s'$ (hence $\mu_1 \in \mathcal{C}_s$) and $\mu_2$ is singular
to all measures in $\mathcal{C}_s'$. It will be shown that $\mu_2 = 0$, from which
it follows that $\mu = \mu_1 \in \mathcal{C}_s$.

For $N, R \in \{ 1, 2, \ldots \}$ let
	$$
	\mathcal{B}_R^N = \left\{ \nu \in \probs\left([-R, R]^d\right); \,
	\widehat\nu(\xi) \leq N |\xi|^{-s / 2} \text{ for all } \xi \right\}.
	$$
These sets are weak-$*$ compact and convex, so Theorem~\ref{septhm} applied
to
	$$
	\mathcal{A}_R = \left\{ \restr{\mu_2}{[-R, R]^d} \right\}, \qquad
	\mathcal{B}_R = \bigcup_{N = 1}^\infty \mathcal{B}_R^N
	$$
gives for each $R$ a Borel set $E_R$ such that $\mu_2(E_R) = 0$ and
	$$
	\nu(\rea^d \setminus E_R) = \nu([-R, R]^d \setminus E_R) = 0
	\quad
	\text{for all } \nu \in \mathcal{B}_R.
	$$
Let
	$$
	E = \bigcup_{R = 1}^\infty E_R.
	$$
Then $\mu_2(E) = 0$ and $\nu(E^c) = 0$ for all $\nu \in \mathcal{C}'_s$ that
have compact support.

If $\nu$ is \emph{any} measure in $\mathcal{C}'_s$, define the measure
$\nu_R$ by $\intd{\nu_R} = \varphi_R \intd{\nu}$, where $\varphi_R$ for
each $R$ is a smooth function such that
$\chi_{[-R, R]^d} \leq \varphi_R \leq \chi_{[-2R, 2R]^d}$.
Then each $\nu_R$ lies in $\mathcal{C}_s'$ by Lemma~\ref{smoothcutlemma} and
$\nu_R$ has compact support, so
	$$
	\nu(E^c) = \lim_{R \to \infty} \nu(E^c \cap [-R, R]^d)
	\leq \lim_{R \to \infty} \nu_R(E^c) = 0.
	$$
This shows that $E^c \in \mathcal{C}_s'^{\perp}$ and hence $\mu(E^c) = 0$.
Since also $\mu_2(E) = 0$, this implies that $\mu_2 = 0$.
\end{proof}

\section{Acknowledgement}
We would like to thank Magnus Aspenberg for helpful discussions.

\bibliographystyle{plain}
\bibliography{references}
\end{document}